\documentclass[12pt]{amsart}
\usepackage{amssymb}
\usepackage{amsmath}
\usepackage{amscd}
\theoremstyle{plain}
\newtheorem{theorem}{Theorem}[section]
\newtheorem{corollary}[theorem]{Corollary}
\newtheorem{proposition}[theorem]{Proposition}
\newtheorem{lemma}[theorem]{Lemma}
{\theoremstyle{remark}

\newtheorem{remark}[theorem]{Remark}}
{\theoremstyle{definition}
\newtheorem{definition}[theorem]{Definition}
\newtheorem{example}[theorem]{Example}}

\newcommand{\rom}{\renewcommand{\labelenumi}{{\rm (\roman{enumi})}}%
\renewcommand{\itemsep}{0pt}}

\newcommand{\N}{\mathbb{N}}
\newcommand{\Z}{\mathbb{Z}}

\newcommand{\C}{\mathbb{C}}
\newcommand{\T}{\mathbb{T}}

\newcommand{\cK}{{\mathcal K}}
\newcommand{\cL}{{\mathcal L}}

\newcommand{\cO}{{\mathcal O}}
\newcommand{\cT}{{\mathcal T}}

\newcommand{\tE}{\widetilde{E}}
\newcommand{\tr}{\tilde{r}}
\newcommand{\F}{{\mathcal F}}

\newcommand{\ip}[2]{\langle\,{#1}\,|\,{#2}\,\rangle}
\newcommand{\s}[3]{{{#1}^{#2}_{\textnormal{#3}}}}
\newcommand{\rs}[1]{{\textnormal{#1}}}

\newcommand{\Ca}{$C^*$-al\-ge\-bra }
\newcommand{\CA}{$C^*$-al\-ge\-bra}
\newcommand{\shom}{$*$-ho\-mo\-mor\-phism }
\newcommand{\shoms}{$*$-ho\-mo\-mor\-phisms }

\DeclareMathOperator{\dom}{dom}

\begin{document}
\title[Topological graphs and SGDSs.]
{Topological graphs and singly generated dynamical systems.}
\author[Takeshi KATSURA]{Takeshi KATSURA}
\address{Department of Mathematics, Keio University,
Yokohama, 223-8522 JAPAN}
\email{katsura@math.keio.ac.jp}
\date{}


\begin{abstract}
In this paper, we introduce the notion of a dual topological graph 
of a given topological graph, and show that it defines a \Ca 
isomorphic to the \Ca of the given one. 
Repeating to take a dual, and taking a projective limit, 
we get a singly generated dynamical system with which the 
associate \Ca is isomorphic to the \Ca of 
the given topological graph. 
This shows that a \Ca of an arbitrary topoloical graph has 
a groupoid model. 
Similar investigation are done for relative topoloical graphs 
and partially defined topoloical graphs which are introduced 
in this paper. 
\end{abstract}

\maketitle

\setcounter{section}{-1}

\section{Introduction}

In \cite{Re2}, 
Renault introduces 
the notion of a singly generated dynamical system (SGDS). 
He associates an \'etale groupoid $G(X,\sigma)$ 
to an SGDS $(X,\sigma)$. 
The \Ca $C^*(X,\sigma)$ of an SGDS $(X,\sigma)$ 
is defined to be the groupoid \Ca $C^*(G(X,\sigma))$ of 
this groupoid $G(X,\sigma)$. 
After that, in \cite{KaI}
the author introduces the notion of a topological graph $E$, 
and associates the \Ca $\cO(E)$ to it. 
In \cite{KaII}, 
the author shows that an SGDS is naturally considered as 
a topological graph, and their \CA s coincide. 
In this paper, we show the converse. 
From a topological graph $E$, 
we can construct an SGDS $(E_\infty^0,\sigma)$ so that 
the \Ca $\cO(E)$ is isomorphic to $C^*(E_\infty^0,\sigma)$ 
(Theorem~\ref{Thm:graph=SGDS}). 
As a corollary, 
the class of C*-algebras of topological graphs coincides with 
the one of SGDSs (Corollary~\ref{Cor:2class}). 
As another corollary, 
every \Ca of a topological graph has a groupoid model. 
This gives another proof of a theorem by Yeend (\cite[Theorem~5.2]{Yee}). 

One of the key ingredients of the constrution is 
the notion of a dual topological graph. 
In \cite{BPRS}, 
Bates, Pask, Raeburn and Szyma\'nski define a dual graph $\hat{E}$ of 
a discrete graph $E$, 
and show that the graph algebra $C^*(\hat{E})$ is isomorphic to $C^*(E)$ 
when $E$ is a row-finite graph with no sinks 
(\cite[Corollary~2.5]{BPRS}). 
We extend their construction to topological graphs 
with slight modification, 
to get a dual graph $E_1$ of a topological graph $E$ 
(Definition~\ref{Def:dual}). 
We show that for arbitrary topological graph $E$, 
$\cO(E_1)$ is isomorphic to $\cO(E)$ (Theorem~\ref{Thm:isom1}). 
Taking a dual topological graph can be repeated many times, 
and we obtain a projective system $((E_k)_{k}, (m_{k,l})_{k,l})$ 
of topological graphs 
such that $E_0$ is a given topological graph $E$, 
and $E_1$ is its dual topological graph. 
Its projective limit $E_\infty$ 
is naturally considered as an SGDS which we are looking for 
(Theorem~\ref{Thm:Einfty=SGDS}). 

In Section~\ref{Sec:groupoid}, 
we compute the groupoid $G(E_\infty^0,\sigma)$, 
in particular its unit space $E_\infty^0$, 
and see that our goupoid here is the same 
as the groupoid $\mathcal{G}_\Lambda$ considered by Yeend in \cite{Yee} 
for the topological $1$-graph $\Lambda$ associated 
with a topological graph $E$. 
In Section~\ref{Sec:Relative}, 
we introduce the notion of a relative topological graph $(E;U)$ 
and its \Ca $\cO(E;U)$. 
Investigation similar to the one for topological graphs 
can be done for relative topological graphs 
to get a groupoid model for the \Ca $\cO(E;U)$ 
of relative topological graphs $(E;U)$. 
By taking $U=\emptyset$, 
we recover \cite[Theorem~5.1]{Yee}. 
In Section~\ref{Sec:partial}, 
we introduce the notion of a partially defined topological graph $E$ 
and its \Ca $\cO(E)$. 
This \Ca naturally arises, for example, in the work of Castro and Kang 
in \cite{CK}. 
We give several examples and show many results 
including a groupoid model of $\cO(E)$.

For a locally compact space $E$, we denote by $C_0(E)$ the \Ca 
of continuos functions on $E$ vanishing at infinity. 
For a locally compact space $E$ and its closed subset $X$, 
we have a surjction $C_0(E) \to C_0(X)$ by restricting functions 
on $E$ to $X$. 
The kernel of this surjection is naturally identified with $C_0(U)$ 
where $U := E\setminus X$ is an open subset of $E$. 
Thus, for an open subset $U$ of a locally compact space $E$ 
(or even for a locally compact space $U$ which is canonically homeomorphic 
to an open subset of $E$), 
we consider $C_0(U)$ as a subalgebra of $C_0(E)$. 

\vspace{0.5cm}
\noindent
{\bfseries Acknowledgments.} 

The author would like to thank his colleague, students, friends 
and family for supporting him until this paper has been finished. 
He is grateful to Gilles G. de~Castro 
for asking him a question. 
Without his question, the author could not finish this paper. 
The content of Section~\ref{Sec:partial} is motivated 
by his question. 
The author would like to thank Trent Yeend and Aidan Sims 
for discussing about groupoids. 
This work was supported by JSPS KAKENHI Grant Number JP18K03345.

\section{Topological graphs}\label{Sec:topgraph}

In this section, we recall the definitions of topological graphs, 
SGDSs, their C*-algebras and their relation. 
For the detail, see \cite{KaI}, \cite{Re2} and \cite{KaII}. 

\begin{definition}
A {\em topological graph} $E=(E^0,E^1,d,r)$ consists of 
two locally compact spaces $E^0$ and $E^1$, 
and two maps $d,r\colon E^1\to E^0$, 
where $d$ is locally homeomorphic 
and $r$ is continuous.
\end{definition}

We regard an element $v$ of $E^0$ as a vertex, 
and an element $e$ of $E^1$ as a directed edge 
pointing from its domain $d(e)\in E^0$
to its range $r(e)\in E^0$. 
%For a topological graph $E=(E^0,E^1,d,r)$, 
%the triple $(E^1,d,r)$ is called 
%a {\em topological correspondence} on $E^0$, 
%which is considered as a generalization of a continuous map. 
By the local homeomorphism $d\colon E^1\to E^0$, 
$E^1$ is ``locally'' isomorphic to $E^0$, 
and the pair $(E^1,d)$ defines a ``domain'' of a continuous map $r$
which is ``locally'' a continuous map from $E^0$ to $E^0$. 

Take a topological graph $E=(E^0,E^1,d,r)$. 
We recall the definition of the C*-algebra $\cO(E)$. 
For the detail, consult \cite{KaI}. 
For $\xi\in C(E^1)$, 
we define a map $\ip{\xi}{\xi}\colon E^0\to [0,\infty]$ by 
$\ip{\xi}{\xi}(v)=\sum_{e\in d^{-1}(v)}|\xi(e)|^2$ for $v\in E^0$.
Then 
\[
C_d(E^1):=\{\xi\in C(E^1)\mid \ip{\xi}{\xi}\in C_0(E^0)\}.
\]
becomes a (right) Hilbert $C_0(E^0)$-module. 
With a left action $\pi_r\colon C_0(E^0)\to\cL(C_d(E^1))$ 
defined by $(\pi_r(f)\xi)(e)=f(r(e))\xi(e)$ 
for $e\in E^1$, $f\in C_0(E^0)$ and $\xi\in C_d(E^1)$, 
$C_d(E^1)$ becomes a C*-correspondence over $C_0(E^0)$. 
The C*-algebra associated with this C*-correspondence $C_d(E^1)$ 
in the sense of \cite{KaPim} 
is the C*-algebra $\cO(E)$. 
In \cite{KaI}, $\cO(E)$ is defined to be the universal \Ca 
generated by a Cuntz-Krieger $E$-pair $(t^0,t^1)$. 
To define a Cuntz-Krieger $E$-pair, it is important to compute 
the largest ideal of $C_0(E^0)$ on which the left action $\pi_r$ 
is injective into $\cK(C_d(E^1))$. 
This ideal is $C_0(\s{E}{0}{rg})$ where 
\begin{align*}
\s{E}{0}{fin}:=\{v\in E^0\mid\mbox{ there exists}&\mbox{ a neighborhood } 
V \mbox{ of } v\\
&\mbox{ such that }
r^{-1}(V)\subset E^1 \mbox{ is compact}\},
\end{align*}
$\s{E}{0}{sce}:=E^0\setminus\overline{r(E^1)}$ 
and $\s{E}{0}{rg} :=\s{E}{0}{fin}\setminus\overline{\s{E}{0}{sce}}$. 
We set $\s{E}{0}{sg}=E^0\setminus \s{E}{0}{rg}$. 
We have $\s{E}{0}{sg}=\s{E}{0}{inf}\cup\overline{\s{E}{0}{sce}}$ 
where $\s{E}{0}{inf}=E^0\setminus \s{E}{0}{fin}$. 
A vertex in $\s{E}{0}{rg}$ is said to be regular, 
and a vertex in $\s{E}{0}{sg}$ is said to be singular. 
We define subsets $\s{E}{1}{rg}$ and $\s{E}{1}{sg}$ of $E^1$ 
by $\s{E}{1}{rg} = d^{-1}(\s{E}{0}{rg})$ 
and $\s{E}{1}{sg} = d^{-1}(\s{E}{0}{sg})$. 
See \cite{KaI} for detail. 

In \cite{Re2}, Renault introduces the following notion. 

\begin{definition}[{\cite[Definition~2.3]{Re2}, see also \cite[Subsection~10.3]{KaII}}]
A {\em singly generated dynamical system} (SGDS) is 
a pair $(X,\sigma)$ where $X$ is a locally compact topological space 
and $\sigma$ is a local homeomorphism from an open subset 
$\dom(\sigma)$ of $X$ onto an open subset $\mbox{ran}(\sigma)$ of $X$. 
\end{definition}

From an SGDS $(X,\sigma)$, 
we have a topological graph $E=(E^0,E^1,d,r)$ 
by setting $E^0=X$, $E^1=\dom(\sigma)$, $d=\sigma$, 
and $r$ is a natural embedding. 
Since $r$ is a natural embedding, 
we have $\s{E}{0}{rg}=\dom(\sigma)$. 

\begin{lemma}
A topological graph $E=(E^0,E^1,d,r)$ 
is given from an SGDS as above 
if and only if $r \colon E^1 \to E^0$ is 
a homeomorphism to an open subset of $E^0$. 
\end{lemma}

\begin{proof}
It is clear that the condition is necessary. 
Assume that $r \colon E^1 \to E^0$ is 
a homeomorphism to an open subset of $E^0$. 
We set $X := E^0$ and 
define a local homeomorphism $\sigma$ 
from an open subset $\dom(\sigma) := r(E^1)$ of $X$ to $X$ 
by $\sigma = d\circ r^{-1}$. 
Then we get an SGDS $(X,\sigma)$ 
with which the associated topological graph
is isomorphic to $E$. 
\end{proof}

As one can see the lemma above, 
an SGDS $(X,\sigma)$ is recovered from 
the topological graph $E$ associated with it. 
Thus we can say that a topological graph $E$ is 
an SGDS if $r$ is a homeomorphism to an open subset of $E^0$. 
In Section~\ref{Sec:infinite}, 
we construct a topological graph $E_\infty$ 
from a given topological graph $E$, 
and show that $E_\infty$ is an SGDS. 

\begin{proposition}[{\cite[Proposition~10.9]{KaII}}]\label{Prop:SGDS=graph}
For an SGDS $(X,\sigma)$, 
the $C^*$-algebra $C^*(X,\sigma)$ is naturally isomorphic to 
$\cO(E)$ for the topological graph $E=(X,\dom(\sigma),\sigma,r)$ 
associated with $(X,\sigma)$ where $r$ is the embedding. 
\end{proposition}

From this Proposition, the class of C*-algebras of SGDSs is 
contained in the one of topological graphs. 
In Section~\ref{Sec:infinite}, 
by showing $\cO(E_\infty) \cong \cO(E)$, 
we show that in fact these two classes coincide.

\section{Dual graphs}\label{Sec:dual}

In \cite{BPRS}, 
Bates, Pask, Raeburn and Szyma\'nski define a dual graph $\hat{E}$ of 
a discrete graph $E$, 
and show that the graph algebra $C^*(\hat{E})$ is isomorphic to $C^*(E)$ 
when $E$ is a row-finite graph with no sinks 
(\cite[Corollary~2.5]{BPRS}). 
In this section, 
we will define a dual graph $E_1$ of 
an arbitrary topological graph $E$, 
and show that $\cO(E_1)$ is isomorphic to $\cO(E)$. 
This graph $E_1$ coincides with the one $\hat{E}$ in \cite{BPRS} 
when $E$ is a row-finite discrete graph with no source 
(no sink in the convention of \cite{BPRS}). 
For a similar construction, see \cite[Section~2]{Br}. 

For a locally compact space $X$, 
the one-point compactification of $X$ is denoted 
by $\widetilde{X}=X\cup\{\infty\}$. 
Note that we use the same symbol $\infty$ 
for one-point compactifications of different spaces. 
This should cause no confusion. 
Note also that even if $X$ is compact, 
we define $\widetilde{X}=X\cup\{\infty\}$ 
and call it the one-point compactification of $X$. 
Thus $\widetilde{X}$ is a compact space 
containing $X$ as an open (not necessarily dense) subset, 
with a special point $\infty$ 
such that $\widetilde{X}=X\cup\{\infty\}$. 

Let $E=(E^0,E^1,d,r)$ be a topological graph. 
We define a topological graph $E_1=(E_1^0,E_1^1,d_1,r_1)$ 
as follows. 

\begin{definition}
We define two subsets $E_1^0 \subset E^0\times \tE^1$ 
and $E_1^1 \subset E^1 \times \tE^1$ by 
\begin{align*}
E_1^0 
&:=\bigl\{(v\,,e)\in E^0\times \tE^1\ \big|\ 
\text{$v=r(e)$ if $e\in E^1$, 
$v\in E^0\setminus\s{E}{0}{rg}$ if $e=\infty$}\bigr\},\\
E_1^1 
&:=\bigl\{(e',e)\in E^1 \times \tE^1\ \big|\ 
\text{$d(e')=r(e)$ if $e\in E^1$, 
$e'\in E^1\setminus\s{E}{1}{rg}$ if $e=\infty$}\bigr\}.
\end{align*}
\end{definition}

Two spaces $E^0\times \tE^1$ and $E^1 \times \tE^1$ 
are open subsets of $\tE^0\times \tE^1$ and $\tE^1 \times \tE^1$, 
respectively. 
%We set $\infty := (\infty,\infty)$ 
%in $\tE^0\times \tE^1$ and $\tE^1 \times \tE^1$. 
We define $\tE_1^0 \subset \tE^0\times \tE^1$ 
and $\tE_1^1 \subset \tE^1 \times \tE^1$ by 
$\tE_1^0 = E_1^0 \cup \{(\infty,\infty)\}$ and 
$\tE_1^1 = E_1^1 \cup \{(\infty,\infty)\}$. 
Then we have the following 

\begin{proposition}\label{Prop:E01cpt}
Two subsets $\tE_1^0 \subset \tE^0\times \tE^1$ 
and $\tE_1^1 \subset \tE^1 \times \tE^1$ satisfy 
\begin{align*}
\tE_1^0 
&=\bigl\{(v\,,e)\in \tE^0\times \tE^1\ \big|\ 
\text{$v=r(e)$ if $e\in E^1$, 
$v\in\tE^0\setminus\s{E}{0}{rg}$ if $e=\infty$}\bigr\},\\
\tE_1^1 
&=\bigl\{(e',e)\in \tE^1 \times \tE^1\ \big|\ 
\text{$d(e')=r(e)$ if $e\in E^1$, 
$e'\in\tE^1\setminus\s{E}{1}{rg}$ if $e=\infty$}\bigr\},
\end{align*}
and are compact. 
Therefore $E_1^0$ and $E_1^1$ are locally compact, 
and their one-point compactifications 
can be naturally identified with $\tE_1^0$ and $\tE_1^1$, 
respectively. 
\end{proposition}

\begin{proof}
We only show that $\tE_1^0$ is compact, 
because one can show that $\tE_1^1$ is compact
in a very similar way. 
Since $\tE^0\times \tE^1$ is compact, 
it suffices to see that $\tE_1^0$ is closed in $\tE^0\times \tE^1$. 
We set $W:=(\tE^0\times \tE^1)\setminus \tE_1^0$, 
and we will show that $W$ is open in $\tE^0\times \tE^1$. 
Take $(v,e)\in W$. 
We are going to find an open subset $W_0$ of $\tE^0\times \tE^1$ 
such that $(v,e)\in W_0\subset W$. 
First consider the case that $e\in E^1$. 
Then $v\neq r(e)$. 
Take open sets $V_1,V_2\subset \tE^0$ such that 
$v\in V_1$, $r(e)\in V_2\subset E^0$, and $V_1\cap V_2=\emptyset$. 
Set $W_0=V_1\times r^{-1}(V_2)$. 
Then $W_0$ is an open set such that $(v,e)\in W_0\subset W$. 
Next consider the case that $e=\infty$. 
We have $v\in\s{E}{0}{rg}$. 
Let $V$ be a compact neighborhood of $v$
with $V\subset \s{E}{0}{rg}$. 
Set $W_0=V\times (\tE^1 \setminus r^{-1}(V))\subset \tE^0\times \tE^1$. 
Since $r^{-1}(V)\subset E^1$ is compact, 
$W_0$ is a neighborhood of $(v,\infty)$.
Since $V\subset \s{E}{0}{rg}$, 
we have $W_0 \subset W$. 
Hence $W$ is an open set. 
This completes the proof. 
\end{proof}

We define two continuoius maps $d_1, r_1\colon E_1^1\to E_1^0$ 
by $d_1((e',e))=(d(e'),e)$ and $r_1((e',e))=(r(e'),e')$ 
for $(e',e)\in E_1^1$. 
These two maps are well-defined and $d_1$ is locally homeomorphic 
because $d$ is locally homeomprphic and 
\[
E_1^1 
=\bigl\{(e',e)\in E^1\times \tE^1\ \big| \ 
(d(e'),e)\in E_1^0\bigr\}. 
\]
Thus we get a topological graph $E_1=(E_1^0,E_1^1,d_1,r_1)$. 

\begin{definition}\label{Def:dual}
The topological graph $E_1=(E_1^0,E_1^1,d_1,r_1)$ is called 
the {\em dual} topological graph of $E$. 
\end{definition}

\begin{lemma}\label{Lem:E1rg}
We have 
$(E_1^0)_{\rs{rg}} = \bigl\{(r(e),e)\in E_1^0 \ \big|\ e\in E^1\bigr\}$
and 
$(E_1^0)_{\rs{sg}} = 
\big\{(v,\infty) \in E_1^0 \ \big|\ v \in \s{E}{0}{sg}\big\}$. 
\end{lemma}

\begin{proof}
First note that 
$\big\{(v,\infty) \in E_1^0 \ \big|\ v \in \s{E}{0}{sg}\big\}
= \big\{(v,e) \in E_1^0 \ \big|\ e=\infty \big\}$
is a closed subset of $E_1^0$ which is homeomorphic to $\s{E}{0}{sg}$. 
Since $\bigl\{(r(e),e)\in E_1^0 \ \big|\ e\in E^1\bigr\}$ is its complement, 
it is an open subset of $E_1^0$. 
The map $r_1\colon E_1^1\to E_1^0$ is the composition of 
the maps $m\colon E_1^1\ni (e',e)\mapsto e'\in E^1$ and 
$\iota \colon E^1 \ni e' \mapsto (r(e'),e') \in E_1^0$. 
The map $m$ is proper beacause 
it is a restriction of the continuous map 
$m_1^1\colon \tE_1^1\ni (e',e)\mapsto e'\in \tE^1$ 
(which will be considered below). 
We are going to show that $m$ is surjective. 
For $e' \in \s{E}{1}{rg}$, 
there exists $e \in E^1$ with $d(e')=r(e)$. 
Then $(e',e) \in E_1^1$ satisfies $m((e',e)) = e'$. 
For $e' \in \s{E}{1}{sg}$, 
$(e',\infty) \in E_1^1$ satisfies $m((e',\infty)) = e'$. 
This shows that $m$ is surjective. 
It is clear that $\iota$ is a homeomorphism. 
Hence the map $r_1\colon E_1^1\to E_1^0$ is 
a proper surjection onto the open subset 
$\bigl\{(r(e),e)\in E_1^0 \ \big|\ e\in E^1\bigr\} \subset E_1^0$. 
This shows that 
$(E_1^0)_{\rs{rg}} = \bigl\{(r(e),e)\in E_1^0 \ \big|\ e\in E^1\bigr\}$. 
Taking their complements, we get 
$(E_1^0)_{\rs{sg}} = 
\big\{(v,\infty) \in E_1^0 \ \big|\ v \in \s{E}{0}{sg}\big\}$. 
\end{proof}

\begin{remark}
By the proof of Lemma~\ref{Lem:E1rg}, 
we see that $E_1^0$ can be devided 
into the open set 
$(E_1^0)_{\rs{rg}}$ homeomorphic to $E^1$ 
and the closed set $(E_1^0)_{\rs{sg}}$ 
homeomorphic to $\s{E}{0}{sg}$. 
Similarly, $E_1^1$ is a union of the open set 
\[
E^2 := \bigl\{(e',e)\in E^1\times E^1\ \big| \ 
d(e')=r(e) \bigr\} \subset E_1^1
\] 
and the closed set 
\[
\big\{(e',\infty) \in E_1^1 \ \big|\ e' \in \s{E}{1}{sg}\big\}
\cong \s{E}{1}{sg}. 
\]
\end{remark}

We define continuous maps $m_1^0$ and $m_1^1$ by 
\[
m_1^0\colon \tE_1^0\ni (v,e)\mapsto v\in \tE^0
\quad \text{ and } \quad 
m_1^1\colon \tE_1^1\ni (e',e)\mapsto e'\in \tE^1. 
\]
For $i=0,1$, 
$(\infty,\infty) \in \tE_1^i$ 
is the only element which is sent to $\infty\in \tE^i$ 
by $m_1^i$. 
Hence $m_1^i$ induces the proper continuous map 
from $E_1^i$ to $E^i$ for $i=0,1$. 
It is easy to check $d(m_1^1(e',e))=m_1^0(d_1(e',e))$ and 
$r(m_1^1(e',e))=m_1^0(r_1(e',e))$ for $(e',e) \in E_1^1$. 
It is also easy to see that if $e' \in E^1$ 
and $(v,e) \in E_1^0$ satisfy $d(e') = m_1^0((v,e))$ 
then $(e',e) \in E_1^1$ is the unique element 
satisfying $m_1^1((e',e)) = e'$ and $d_1((e',e)) = (v,e)$.

For the definition of factor maps and their regularity, 
see \cite{KaII}. 
We call a factor map $m = (m^0,m^1)$ {\em surjective} 
if $m^0$ is surjective. 
In this case $m^1$ is also surjective.

\begin{proposition}\label{Prop:factormap}
The pair $m_1=(m_1^0,m_1^1)$ 
is a regular surjective factor map from $E_1$ to $E$. 
\end{proposition}

\begin{proof}
We have already shown that $m_1$ is a factor map. 
To show that it is regular, 
it suffices to show 
$(m_1^0)^{-1}(\s{E}{0}{rg}) \subset (E_1^0)_{\rs{rg}}$
because $(m_1^1)^{-1}(E^1)=E_1^1$. 
This follows from Lemma~\ref{Lem:E1rg}. 
Finally, one can show that $m_1^0$ is surjective 
in a similar way to the proof of Lemma~\ref{Lem:E1rg}. 
\end{proof}

\section{Iterating the construction of the dual graphs}
\label{Sec:iterating}

In this section, 
we define a sequence of 
topological graphs $E_k=(E_k^0,E_k^1,d_k,r_k)$ 
for $k \in \N:=\{0,1,2,\ldots\}$ 
and factor maps $m_{k,l}$ from $E_l$ to $E_k$ for $k \leq l$ 
so that $((E_k)_{k}, (m_{k,l})_{k,l})$ 
becomes a projective system (see \cite[Definition~4.1]{KaII}). 

Let $E=(E^0,E^1,d,r)$ be a topological graph, 
and $E_1=(E_1^0,E_1^1,d_1,r_1)$ 
be the dual topological graph of $E$ 
defined in Section~\ref{Sec:dual}. 
Fix an integer $k$ grater than $1$, 
and we will define a topological graph $E_k=(E_k^0,E_k^1,d_k,r_k)$. 

We define $E_k^0$ and $E_k^1$ by 
\begin{align*}
E_k^0 
&:= \bigl\{\, (v,\, e_1, e_2, \ldots, e_k) 
\in E^0 \times \tE^1 \times \cdots \times \tE^1
\ \big|\ (v, e_1) \in E_1^0, (e_i,e_{i+1}) \in \tE_1^1\bigr\} \\
E_k^1 
&:= \bigl\{ (e_0, e_1, e_2, \ldots, e_k) 
\in E^1 \times \tE^1 \times \cdots \times \tE^1
\ \big|\ (e_0, e_1) \in E_1^1, (e_i,e_{i+1}) \in \tE_1^1\bigr\}. 
\end{align*}

\begin{lemma}
Two sets 
\begin{align*}
\tE_k^0 
&:= E_k^0 \cup \{(\infty,\ldots,\infty)\} \\
&\phantom{:}= \bigl\{\, (v,\, e_1, e_2, \ldots, e_k) 
\in \tE^0 \times \tE^1 \times \cdots \times \tE^1
\ \big|\ (v, e_1) \in \tE_1^0, (e_i,e_{i+1}) \in \tE_1^1\bigr\} \\
\tE_k^1 
&:= E_k^1 \cup \{(\infty,\ldots,\infty)\} \\
&\phantom{:}= \bigl\{ (e_0, e_1, e_2, \ldots, e_k) 
\in \tE^1 \times \tE^1 \times \cdots \times \tE^1
\ \big|\ (e_i,e_{i+1}) \in \tE_1^1\bigr\} 
\end{align*}
are compact. 
Therefore $E_k^0$ and $E_k^1$ are locally compact, 
and their one-point compactifications 
can be naturally identified with $\tE_k^0$ and $\tE_k^1$, 
respectively. 
\end{lemma}

\begin{proof}
Since $\tE_1^0 \subset \tE^0 \times \tE^1$ 
and $\tE_1^1 \subset \tE^1 \times \tE^1$ are closed, 
$\tE_k^0$ and $\tE_k^1$ are closed subsets of 
$\tE^0 \times \tE^1 \times \cdots \times \tE^1$ and 
$\tE^1 \times \tE^1 \times \cdots \times \tE^1$, respectively. 
Hence they are compact. 
\end{proof}

We define $d_k, r_k \colon E_k^1 \to E_k^0$ by 
\begin{align*}
d_k\big((e_0, e_1, e_2, \ldots, e_k)\big)
&=(d(e_0), e_1, e_2, \ldots, e_k),\\ 
r_k\big((e_0, e_1, e_2, \ldots, e_k)\big)
&=(r(e_0), e_0, e_1, \ldots, e_{k-1}) 
\end{align*}
for $(e_0, e_1, e_2, \ldots, e_k) \in E_k^1$. 

\begin{lemma}\label{Lem:dkrk}
The map $d_k$ is a well-defined local homeomorphism, 
and $r_k$ is a well-defined continuous map. 
\end{lemma}

\begin{proof}
These follow from the facts 
that $d_1\colon E_1^1 \to E_1^0$ is locally homeomorphic, 
and that $r_1\colon E_1^1\to E_1^0$ is continuous. 
\end{proof}

Thus we get the topological graph $E_k=(E_k^0,E_k^1,d_k,r_k)$ 
for $k=2,3,\ldots$. 
Next, we define maps $m_k^0$ and $m_k^1$ by 
\begin{align*}
m_k^0 \colon 
E_k^0\ni (v,\, e_1, e_2, \ldots, e_k) 
& \mapsto v \in E^0 \\
m_k^1 \colon 
E_k^1\ni (e_0, e_1, e_2, \ldots, e_k) 
& \mapsto e_0 \in E^1
\end{align*}
We set $m_k := (m_k^0, m_k^1)$ 
which will be shown to be a regular factor map. 

We set $E_0 := E$, and $m_{0,k} := m_{k}$. 
For $k \leq l$, 
we define 
\begin{align*}
m_{k,l}^0 \colon 
E_l^0\ni (v,\, e_1, e_2, \ldots, e_l) 
& \mapsto (v,\, e_1, e_2, \ldots, e_k) \in E_k^0 \\
m_{k,l}^1 \colon 
E_l^1\ni (e_0, e_1, e_2, \ldots, e_l) 
& \mapsto (e_0, e_1, e_2, \ldots, e_k) \in E_k^1
\end{align*}
and set $m_{k,l} := (m_{k,l}^0, m_{k,l}^1)$. 
The following is easy to see. 

\begin{lemma}
For $j \leq k \leq l$ and $i=0,1$, 
we have $m_{j,k}^i \circ m_{k,l}^i = m_{j,l}^i$. 
\end{lemma}

One can show that $m_{k,l}$ is a regular factor map 
in a similar way to the proof of Proposition~\ref{Prop:factormap}.
Instead, we will use the following proposition. 

\begin{proposition}\label{Prop:k+j}
There exist natural 
isomorphisms $(E_k)_j \cong E_{k+j}$ 
for $j,k \in \N$. 
Under these isomorphisms, 
for $j,k,l \in \N$ with $j \leq l$ and $i=0,1$ 
the maps $m_{j,l}^i \colon (E_k)_l^i \to (E_k)_j^i$ 
defined as above using $E_k$ instead of $E$ 
coincide with $m_{k+j,k+l}^i \colon E_{k+l}^i \to E_{k+j}^i$. 
\end{proposition}

\begin{proof}
For $j,k \in \N$, 
we define the map $E_{k+j}^0 \to (E_k)_j^0$ by 
\[
(v_0, e_1, e_2, \ldots, e_{k+j})
\mapsto \big( (v_0, e_1, \ldots, e_k), (e_1, e_2, \ldots, e_{k+1}), 
\ldots, (e_j, e_{j+1}, \ldots, e_{k+j}) \big), 
\]
and $E_{k+j}^1 \to (E_k)_j^1$ similarly. 
It is routine to check that 
these maps induce the isomorphism $(E_k)_j \cong E_{k+j}$ 
with desired properties. 
\end{proof}

\begin{proposition}
For every $k, l \in \N$ with $k \leq l$, 
$m_{k,l} := (m_{k,l}^0, m_{k,l}^1)$ 
is a regular surjective factor map from $E_l$ to $E_k$. 
\end{proposition}

\begin{proof}
By Proposition~\ref{Prop:factormap}, 
$m_{0,1}=m_1$ satisfies the desired properties. 
By Proposition~\ref{Prop:k+j}, 
$m_{k,k+1}$ coincides with $m_{0,1}$ for $E_k$, 
and hence satisfies the desired properties. 
Finally, since the desired properties are stable 
under taking compositions, 
\[
m_{k,l} = m_{k,k+1} \circ m_{k+1,k+2} \circ \cdots \circ m_{l-1,l}
\]
satisfies the desired properties 
for every $k, l \in \N$ with $k \leq l$. 
\end{proof}

Thus we get the regular surjective 
projective system $((E_k)_{k}, (m_{k,l})_{k,l})$ 
of topological graphs 
such that $E_0$ is a given topological graph $E$, 
and $E_1$ is its dual topological graph.

\section{Infinite path spaces and SGDSs}\label{Sec:infinite}

In this section, 
we define a topological graph 
$E_\infty=(E_\infty^0, E_\infty^1, d_\infty, r_\infty)$ 
and show that $E_\infty$ is the project limit 
of $((E_k)_{k}, (m_{k,l})_{k,l})$. 

We define $E_\infty^0$ and $E_\infty^1$ by 
\begin{align*}
E_\infty^0 
&:= \bigl\{\, (v,\, e_1, e_2, \ldots) 
\in E^0 \times \tE^1 \times \cdots 
\ \big|\ (v, e_1) \in E_1^0, (e_i,e_{i+1}) \in \tE_1^1\bigr\} \\
E_\infty^1 
&:= \bigl\{ (e_0, e_1, e_2, \ldots) 
\in E^1 \times \tE^1 \times \cdots 
\ \big|\ (e_0, e_1) \in E_1^1, (e_i,e_{i+1}) \in \tE_1^1\bigr\}. 
\end{align*}
Similarly as in Section~\ref{Sec:iterating}, 
we see that 
\begin{align*}
\tE_\infty^0 
&:= E_\infty^0 \cup \{(\infty,\infty,\ldots)\}\\
&\phantom{:}= \bigl\{\, (v,\, e_1, e_2, \ldots) 
\in \tE^0 \times \tE^1 \times \cdots 
\ \big|\ (v, e_1) \in \tE_1^0, (e_i,e_{i+1}) \in \tE_1^1\bigr\} 
\end{align*}
and
\begin{align*}
\tE_\infty^1 
&:= E_\infty^1 \cup \{(\infty,\infty,\ldots)\}\\
&\phantom{:}= \bigl\{ (e_0, e_1, e_2, \ldots) 
\in \tE^1 \times \tE^1 \times \cdots 
\ \big|\ (e_i,e_{i+1}) \in \tE_1^1\bigr\}\phantom{kkkkkkkkkk} 
\end{align*}
are compact. 
Hence $E_\infty^0$ and $E_\infty^1$ are locally compact spaces 
whose one-point compactifications 
can be naturally identified with $\tE_\infty^0$ and $\tE_\infty^1$. 
Next we define $d_\infty, r_\infty \colon E_\infty^1 \to E_\infty^0$ by 
\begin{align*}
d_\infty\big((e_0, e_1, e_2, \ldots )\big)
&=(d(e_0), e_1, e_2, \ldots ),\\ 
r_\infty\big((e_0, e_1, e_2, \ldots )\big)
&=(r(e_0), e_0, e_1, \ldots ) 
\end{align*}
for $(e_0, e_1, e_2, \ldots) \in E_\infty^1$. 
As in Lemma~\ref{Lem:dkrk}, 
we can see that $d_\infty$ is a well-defined local homeomorphism, 
and $r_\infty$ is a well-defined continuous map. 
Thus we get a topological graph 
$E_\infty := (E_\infty^0,E_\infty^1,d_\infty,r_\infty)$. 

\begin{theorem}\label{Thm:Einfty=SGDS}
The topological graph 
$E_\infty=(E_\infty^0,E_\infty^1,d_\infty,r_\infty)$ 
is an SGDS. 
\end{theorem}

\begin{proof}
We need to show $r_\infty\colon E_\infty^1 \to E_\infty^0$ 
is a homeomorphism onto the image, 
but it is clear. 
\end{proof}

Note that the image of $r_\infty$ 
which coincides with $(E_\infty^0)_{\rs{rg}}$
is the set of elements $(v,e_1, e_2, \ldots) \in E_\infty^0$ 
with $e_1 \neq \infty$. 
Taking complements, we get 
\begin{align*}
(E_\infty^0)_{\rs{sg}}
&= \big\{ (v, e_1, e_2, \ldots) \in E_\infty^0\ \big|\ 
e_1= \infty \big\} \\
&= \big\{ (v, \infty, \infty, \ldots) \in E_\infty^0\ \big|\ 
v \in \s{E}{0}{sg} \big\}
\cong \s{E}{0}{sg}.
\end{align*}

It is routine to check the following. 
For the definition of project limits, 
see \cite[Definition~4.4]{KaII}. 

\begin{proposition}\label{Prop:lim}
The topological graph $E_\infty$ 
coincided with the project limit of the 
projective system $((E_k)_{k}, (m_{k,l})_{k,l})$ 
defined in Section~\ref{Sec:iterating}. 
\end{proposition}

For $k \in \N$, 
the natural factor map $m_{k,\infty} = (m_{k,\infty}^0, m_{k,\infty}^1)$ 
from $E_\infty$ to $E_k$ 
coming from Proposition~\ref{Prop:lim} is regular and surjective, 
and can be expressed as 
\begin{align*}
m_{k,\infty}^0\big( (v,\, e_1, e_2, \ldots) \big)
& = (v,\, e_1, e_2, \ldots, e_k) \\
m_{k,\infty}^1 \big( (e_0, e_1, e_2, \ldots) \big)
& = (e_0, e_1, e_2, \ldots, e_k). 
\end{align*}
We denote by $m_\infty$ 
the factor map
$m_{0,\infty}$ from $E_\infty$ to $E = E_0$.

\begin{proposition}\label{Prop:WhenSGDS}
For a topological graph $E$,
the following coonditions are equivalent:
\begin{enumerate}
\rom
\item $E$ is an SGDS, 
\item the factor map $m_1$ from $E_1$ to $E$ is an isomorphism, 
\item the factor map $m_k$ from $E_k$ to $E$ is an isomorphism 
for some $k\geq 1$, 
\item the factor map $m_k$ from $E_k$ to $E$ is an isomorphism 
for every $k\in \N$, 
\item the factor map $m_\infty$ from $E_\infty$ to $E$ is an isomorphism, 
\end{enumerate}
\end{proposition}

\begin{proof}
(i)$\Rightarrow$(ii): 
Suppose $E$ is an SGDS. 
We show that the surjective proper map $m_1^0$ is injective. 
To do so, take $v \in E^0$ and $e,e' \in \tE^1$ with $(v,e),(v,e')\in E_1^0$. 
The goal is to show $e=e'$. 
First consider the case $v \in r(E^1)$. 
Since $r(E^1) = \s{E}{0}{rg}$, we have $e,e' \in E^1$ and $r(e)=r(e')=v$. 
Since $r$ is injective, we have $e=e'$. 
Next consider the case $v \notin r(E^1)$. 
In this case, $e=e'=\infty$. 
Thus we have shown that $m_1^0$ is injective. 
Injectivity of $m_1^1$ can be shown similarly. 
These show that the factor map $m_1$ is an isomorphism. 

(ii)$\Rightarrow$(iii): This is trivial. 

(iii)$\Rightarrow$(ii): 
For some $k\geq 1$, 
the factor map $m_k$ is the composition of 
the surjective factor map $m_{1,k}$ and 
the factor map $m_1$. 
Hence when $m_k$ is an isomorphism, 
$m_1$ is also an isomorphism. 

(ii)$\Rightarrow$(iv): 
Suppose $m_1$ is an isomorphism. 
Then $m_{1,2}$ is also an isomorphism 
because it is $m_1$ for $E_1$ which is isomorphic to $E$. 
Repeating this argument, we will see that $m_{k,k+1}$ is an isomorphism 
for every $k\in \N$. 
Therefore $m_k=m_{0,1}\circ m_{1,2}\circ \cdots \circ m_{k-1,k}$ is 
also an isomorphism. 

(iv)$\Rightarrow$(v): 
This is an easy fact on projective limit. 

(v)$\Rightarrow$(i): This follows from Theorem~\ref{Thm:Einfty=SGDS}. 
\end{proof}

\begin{corollary}
For a topological graph $E$,
$(E_\infty)_k$ for $k \in \N$ and $(E_\infty)_\infty$ 
are naturally isomorphic to $E_\infty$. 
\end{corollary}

\begin{proof}
This follows from Theorem~\ref{Thm:Einfty=SGDS}
and Proposition~\ref{Prop:WhenSGDS}. 
\end{proof}

%Markov chain. 

\section{$C^*$-algebras}\label{Sec:C*alg}

In this section, 
we show that topological graphs constructed in 
Sections~\ref{Sec:dual}, {Sec:iterating} and {Sec:infinite}
define \CA s naturally isomorphic to the \Ca $\cO(E)$ 
of the original topological graph $E$. 

\begin{theorem}\label{Thm:isom1}
The injective \shom $\mu\colon \cO(E)\to \cO(E_1)$ 
induced by the factor map $m_1=(m_1^0, m_1^1)$ 
as in \cite[Proposition~2.9]{KaII} 
is an isomorphism. 
\end{theorem}

\begin{proof}
Since $m_1^0$ is surjective, 
$\mu$ is injective. 
To show that $\mu$ is isomorphic 
it suffices to show that its image 
contains the generators $t_1^0(C_0(E_1^0))$ 
and $t_1^1(C_{d_1}(E_1^1))$ of $\cO(E_1)$ 
where $(t_1^0, t_1^1)$ is the universal Cuntz-Krieger $E_1$-pair. 
For the universal Cuntz-Krieger $E$-pair $(t^0, t^1)$, 
the map $\mu$ satisfies $\mu(t^0(g))=t_1^0(g\circ m_1^0)$ and  
$\mu(t^1(\eta))=t_1^1(\eta\circ m_1^1)$ for $g \in C_0(E^0)$ and 
$\eta \in C_d(E^1)$. 

Take $f \in C_0(E_1^0)$. 
We can find $g \in C_0(E^0)$ such that 
$g(v)=f(v,\infty)$ for all $v \in \s{E}{0}{sg}$. 
Then we have (unique) $h \in C_0(E^1)$ 
such that $h \circ m_1^1= f - g \circ m_1^0$. 
By Lemma \ref{Lem:E1rg}, 
we have $h \circ m_1^1 \in C_0((E_1^0)_{\rs{rg}})$. 
Hence we get 
\[
t_1^0(h \circ m_1^1)=\psi_{t_1^1}\big(\pi_{r_1}(h \circ m_1^1)\big)
=\psi_{t_1^1}\big(\psi(\pi_r(h))\big)
=\mu\big(\psi_{t^1}(\pi_r(h))\big)
\]
by \cite[Proposition~2.5]{KaII} (see \cite{KaII} for the notation). 
Therefore 
\[
t_1^0(f) = t_1^0(g \circ m_1^0)+t_1^0(h \circ m_1^1) 
= \mu(t^0(g))+\mu\big(\psi_{t^1}(\pi_r(h))\big)
\]
is in the image of $\mu$. 

Next we see that $t_1^1(C_{d_1}(E^2))$ is in the image of $\mu$. 
Take $f,g \in C_c(E^1)$ and define $h \in C_c(E^2)$ 
by $h(e,e')=f(e)g(e')$. 
We have $t_1^1(h)=t_1^1(f\circ m_1^1)t_1^0(g)$
where $g$ is considered as an element in $C_0(E_1^0)$ 
through the map $E^1\ni e \mapsto (r(e),e) \in E_1^0$ 
which is homeomorphic to an open subset. 
By the fact $t_1^1(f\circ m_1^1) = \mu(t^1(f))$ 
and the former part of this proof, 
we have that $t_1^1(h)$ is in the image of $\mu$. 
Since a function such as $h$ spans a dense subset of 
$C_c(E^2)$ with respect to the supremum norm, 
$t_1^1(C_{d_1}(E^2))$ is in the image of $\mu$ by \cite[Lemma~1.26]{KaI}.

Now take $\xi \in C_{d_1}(E_1^1)$, 
and we will show $t_1^1(\xi)$ is in the image of $\mu$. 
We have a surjection $C_{d_1}(E_1^1) \to C_{d}(\s{E}{1}{sg})$ 
induced by the map $\s{E}{1}{sg} \ni e\mapsto (e,\infty)\in E_1^1$
whose kernel is identified with $C_{d_1}(E^2)$. 
The composition of the map 
$C_d(E^1) \ni \eta \mapsto \eta \circ m_1^1 \in C_{d_1}(E_1^1)$ 
and $C_{d_1}(E_1^1) \to C_{d}(\s{E}{1}{sg})$ 
is the natural restriction map $C_d(E^1) \to C_{d}(\s{E}{1}{sg})$
which is surjective by \cite[Lemma~1.11]{KaI}. 
Therefore there exist 
$\eta \in C_d(E^1)$ and $\zeta \in C_{d_1}(E^2)$ 
such that $\xi = \eta \circ m_1^1 + \zeta$. 
As we have seen above, $t_1^1(\zeta)$ is 
in the image of $\mu$. 
Since $t_1^1(\eta\circ m_1^1) = \mu(t^1(\eta))$ is 
also in the image of $\mu$, 
$t_1^1(\xi)$ is 
in the image of $\mu$. 
This completes the proof. 
\end{proof}

\begin{theorem}\label{Thm:isomEk}
For each $k \in \N$, 
the injective \shom $\mu\colon \cO(E)\to \cO(E_k)$ 
induced by the factor map $m_k=(m_k^0,m_k^1)$ 
is an isomorphism. 
\end{theorem}

\begin{proof}
For each $k \in \N$, 
the dual graph of $E^k$ is $E^{k+1}$. 
Hence by Proposition~\ref{Thm:isom1}, 
the natural map $\cO(E_k)\to \cO(E_{k+1})$ is an isomorphism. 
Since the injective \shom $\mu\colon \cO(E)\to \cO(E_k)$ 
is a composition of these maps, 
it is an isomorphism. 
\end{proof}

\begin{theorem}\label{Thm:isomEinf}
The injective \shom $\mu\colon \cO(E)\to \cO(E_\infty)$ 
induced by the factor map $m_\infty=(m_\infty^0,m_\infty^1)$ 
is an isomorphism. 
\end{theorem}

\begin{proof}
This follows from Theorem~\ref{Thm:isomEk}. 
\end{proof}

\begin{theorem}\label{Thm:graph=SGDS}
For a topological graph $E$, the \Ca $\cO(E)$ is 
isomorphic to the \Ca $C^*(E_\infty^0,\sigma)$ 
of the SGDS $(E_\infty^0,\sigma)$ where $\dom(\sigma)=r_\infty(E_\infty^1)$ 
and $\sigma = d_\infty \circ r_\infty^{-1}$. 
\end{theorem}

\begin{proof}
This follows from Theorem~\ref{Thm:isomEinf} and 
Proposition~\ref{Prop:SGDS=graph}. 
\end{proof}

\begin{corollary}\label{Cor:2class}
The class of C*-algebras of topological graphs coincides with 
the one of SGDSs. 
\end{corollary}

\section{Groupoid models}\label{Sec:groupoid}

Let $E=(E^0,E^1,d,r)$ be a topological graph. 
In this section, 
we investigate the groupoid model of $\cO(E)$. 

By Theorem~\ref{Thm:graph=SGDS}, 
the \Ca $\cO(E)$ is isomorphic to the \Ca $C^*(E_\infty^0,\sigma)$ 
of the SGDS $(E_\infty^0,\sigma)$ where $\dom(\sigma)=r_\infty(E_\infty^1)$ 
and $\sigma = d_\infty \circ r_\infty^{-1}$. 
The \Ca $C^*(E_\infty^0,\sigma)$ is, by definition, the groupoid \Ca 
of the \'etale groupoid 
\begin{align*}
G(E_\infty^0,\sigma)
:= \{(v,m-n,v')\in E_\infty^0 \times \Z \times E_\infty^0 \mid
m&,n \in \N,\ v \in \dom (\sigma^m),\\ 
& v' \in \dom (\sigma^n),\ 
\sigma^m(v)=\sigma^n(v')\}.
\end{align*}
See \cite{Re2} for the groupoid structure and the topology 
of $G(E_\infty^0,\sigma)$ which is called the Renault-Deaconu groupoid. 
We can express this groupoid $G(E_\infty^0,\sigma)$ using $d_\infty$ and 
$r_\infty$ instead of $\sigma$ as follows. 
For $m,n \in \N$, define $E_\infty(m,n)$ to be the set of pairs 
$(v,v')$ of two elements $v,v'\in E_\infty^0$ such that 
there exist $e_1,e_2,\ldots,e_m, e'_1,e'_2,\ldots,e'_n \in E_\infty^1$ 
satisfying
\begin{align*}
v&=r_\infty(e_1),\ d_\infty(e_1)=r_\infty(e_2),\ldots, 
d_\infty(e_{m-1})=r_\infty(e_m),\\
v'&=r_\infty(e'_1),\ d_\infty(e'_1)=r_\infty(e'_2),\ldots, 
d_\infty(e'_{n-1})=r_\infty(e'_n),\ 
d_\infty(e_m)=d_\infty(e'_n).
\end{align*}
Note that since $r_\infty$ is injective, 
the sequence of edges above is, if it exists, unique. 
For $n,m$ grater than $0$, 
a pair $(v,e_1, e_2, \ldots)$ and $(v', e'_1, e'_2, \ldots)$ of elements 
in $E_\infty^0$ is in $E_\infty(m,n)$ if and only if 
$e_{m-1},e_{n-1}' \neq \infty$ and 
\[
(d(e_{m-1}),e_m,e_{m+1},\ldots)=(d(e'_{n-1}),e'_n,e'_{n+1},\ldots).
\]
When either $m$ or $n$ is $0$, a similar discription of $E_\infty(m,n)$ 
is possible. 
Using these spaces $E_\infty(m,n)$, 
we have 
\begin{align*}
G(E_\infty^0,\sigma)
& = \{(v,m-n,v')\in E_\infty^0 \times \Z \times E_\infty^0 \mid
m,n \in \N,\ (v, v') \in E_\infty(m,n)\}.
\end{align*}

The unit space of the groupoid $G(E_\infty^0,\sigma)$ is identified with 
$E_\infty^0$. 
We are going to see that this set $E_\infty^0$ without topology 
can be considered as the set of negative orbits in $E$ 
as defined in \cite[Definition~4.6]{KaIII}. 

For $n=2,3,\ldots$, 
we define a space $E^n$ of {\em paths} with {\em length} $n$ by 
\[
E^n :=\{(e_1,e_2,\ldots,e_n)\in E^1\times E^1 \times \cdots\times E^1
\mid d(e_{k})=r(e_{k+1})\ (1\leq k\leq n-1)\}.
\]
We denote by $E^*$ the disjoint union $\bigcup_{n=0}^\infty E^n$. 
We define $d\colon E^* \to E^0$ by the identity on $E^0$, $d$ on $E^1$ and 
$d(\alpha):=d(e_n)$ for $\alpha=(e_1,e_2,\ldots,e_n)\in E^n$. 
We set $\s{E}{n}{sg}:=\{\alpha \in E^n\mid d(\alpha) \in \s{E}{0}{sg}\}$ 
for $n=0,1,\ldots$ and 
$\s{E}{*}{sg}:=\{\alpha \in E^*\mid d(\alpha) \in \s{E}{0}{sg}\}
=\bigcup_{n=0}^\infty \s{E}{n}{sg}$. 
We also define a space $E^\infty$ of {\em infinite paths} by 
\[
E^\infty :=\{(e_1,e_2,\ldots,e_n,\ldots)\in 
E^1\times E^1 \times \cdots\times E^1 \times \cdots
\mid d(e_{k})=r(e_{k+1})\ (1\leq k)\}.
\]

\begin{proposition}
The set $E_\infty^0$ is isomorphic as a set to the disjoint union 
$\s{E}{*}{sg} \cup E^\infty$. 
\end{proposition}

\begin{proof}
We define a map $\s{E}{0}{sg} \to E_\infty^0$ by 
$v\mapsto (v,\infty,\infty,\ldots,\infty,\ldots)$. 
For $n=1,2,\ldots$, 
we define a map $\s{E}{n}{sg} \to E_\infty^0$ 
by $(e_1,e_2,\ldots,e_n) \mapsto (r(e_1),e_1,e_2,\ldots,e_n,\infty,\ldots,\infty,\ldots)$. 
We also define a map $E^\infty \to E_\infty^0$ 
by $(e_1,e_2,\ldots,e_n,\ldots) \mapsto (r(e_1),e_1,e_2,\ldots,e_n,\ldots)$. 
These maps induce the bijection. 
\end{proof}

The set $\s{E}{*}{sg} \cup E^\infty$ is nothing but the 
set of negative orbits in $E$ as defined in \cite[Definition~4.6]{KaIII}. 
By a similar map, 
we get a bijection between $E_\infty^1$ and the disjoint union 
$\bigcup_{n=1}^\infty \s{E}{n}{sg} \cup E^\infty$. 
Through these bijections, the map $r_\infty$ is the natural embedding, 
and the map $d_\infty$ is the disjoint union of the left shift 
\[
e_1 \mapsto d(e_1),\ (e_1,e_2,\ldots,e_n)\mapsto (e_2,\ldots,e_n),\ 
(e_1,e_2,\ldots,e_n,\ldots) \mapsto (e_2,\ldots,e_n,\ldots). 
\]
The groupoid $G(E_\infty^0,\sigma)$ can also be identified with the set 
\begin{align*}
\{&(\alpha,\beta) \in \s{E}{*}{sg} \times \s{E}{*}{sg} 
\mid d(\alpha)=d(\beta)\} \\
&\cup \{((e_1,e_2,\ldots,e_n,\ldots),(e'_1,e'_2,\ldots,e'_n,\ldots)) 
\in E^\infty \times E^\infty \mid \exists n,m \forall k, 
e_{n+k}=e'_{m+k}\}.
\end{align*}

The set $\s{E}{*}{sg} \cup E^\infty$ is nothing but the set of 
boundary path defined by Kumjian and Li in 
\cite[Definition~4.1]{KL}. 
As shown in \cite[Proposition~4.6]{KL}, 
this set coincides with the set of 
boundary path defined by Yeend in 
\cite[Section~4]{Yee}. 
Hence the groupoid $G(E_\infty^0,\sigma)$ 
coincides with the boundary-path groupoid $\mathcal{G}_\Lambda$ 
defined in \cite[Definition~4.1]{Yee} 
for the topological $1$-graph $\Lambda$ associated with $E$. 
Thus we give another proof for \cite[Theorem~5.2]{Yee}. 
(See also \cite[Theorem~7.7]{KL}.)

\section{Relative topological graphs}\label{Sec:Relative}

Let $E=(E^0,E^1,d,r)$ be a topological graph. 
Let $U$ be an open subset of $\s{E}{0}{rg}$. 
We call a pair $(E;U)$ a relative topological graph. 
From a relative topological graph, 
we get a \Ca $\cO(E;U)$ as a relative Pimsner algebra 
of $C_d(E^1)$ with respect to the ideal $C_0(U) \subset C_0(\s{E}{0}{rg})$ 
as defined in \cite[Definition 2.18]{MS}. 
When $U=\s{E}{0}{rg}$, 
we recover original topological graph \Ca $\cO(E)=\cO(E;\s{E}{0}{rg})$. 
When $U=\emptyset$, 
we get a Toeplitz-type \Ca $\cT(E)=\cO(E;\emptyset)$ 
(see \cite[Definition~2.2]{KaI}). 

By \cite[Proposition~3.21]{KaII}, 
the \Ca $\cO(E;U)$ of a relative topological graph $(E;U)$ 
is isomorphic to $\cO(E_Y)$ defined in \cite[Section~3]{KaII} 
for $Y=\s{E}{0}{rg}\setminus U$. 
We do not use this fact. 

We define a relative dual topological graph $E_U=(E_U^0,E_U^1,d_U,r_U)$ 
as follows. 

\begin{definition}
We define two subsets $E_U^0 \subset E^0\times \tE^1$ 
and $E_U^1 \subset E^1 \times \tE^1$ by 
\begin{align*}
E_U^0 
&:=\bigl\{(v\,,e)\in E^0\times \tE^1\ \big|\ 
\text{$v=r(e)$ if $e\in E^1$, 
$v\in E^0\setminus U$ if $e=\infty$}\bigr\},\\
E_U^1 
&:=\bigl\{(e',e)\in E^1 \times \tE^1\ \big|\ 
\text{$d(e')=r(e)$ if $e\in E^1$, 
$e'\in E^1\setminus d^{-1}(U)$ if $e=\infty$}\bigr\}.
\end{align*}
\end{definition}

By a similar way to the proof of Proposition~\ref{Prop:E01cpt}, 
one can show that 
\begin{align*}
\tE_U^0 
&:=E_U^0 \cup \{(\infty,\infty)\} \\
&\phantom{:}= \bigl\{(v\,,e)\in \tE^0\times \tE^1\ \big|\ 
\text{$v=r(e)$ if $e\in E^1$, 
$v\in\tE^0\setminus U$ if $e=\infty$}\bigr\},\\
\tE_U^1 
&:=E_U^1 \cup \{(\infty,\infty)\} \\
&\phantom{:}= \bigl\{(e',e)\in \tE^1 \times \tE^1\ \big|\ 
\text{$d(e')=r(e)$ if $e\in E^1$, 
$e'\in\tE^1\setminus d^{-1}(U)$ if $e=\infty$}\bigr\}
\end{align*}
are compact. 
The key fact is that for $(v,e) \in \tE^0\times \tE^1\setminus \tE_U^0$ 
with $e=\infty$, we have $v \in U \subset \s{E}{0}{rg}$. 
Hence $E_U^0$ and $E_U^1$ are locally compact, 
and their one-point compactifications 
can be naturally identified with $\tE_U^0$ and $\tE_U^1$, 
respectively. 
We define two continuoius maps $d_U, r_U\colon E_U^1\to E_U^0$ 
by $d_U((e',e))=(d(e'),e)$ and $r_U((e',e))=(r(e'),e')$ 
for $(e',e)\in E_U^1$. 
These two maps are well-defined and $d_U$ is locally homeomorphic 
because 
\[
E_U^1 
=\bigl\{(e',e)\in E^1\times \tE^1\ \big| \ 
(d(e'),e)\in E_U^0\bigr\}. 
\]
Thus we get a topological graph $E_U=(E_U^0,E_U^1,d_U,r_U)$. 
which is called the {\em relative dual} topological graph of $E$ 
with respect to $U$. 

In the exactly same way to the proof of Lemma~\ref{Lem:E1rg}, 
one can show that the map $r_U\colon E_U^1\to E_U^0$ is 
a proper surjection onto the open subset 
$\bigl\{(r(e),e)\in E_U^0 \ \big|\ e\in E^1\bigr\} \subset E_U^0$. 
This shows that 
$(E_U^0)_{\rs{rg}} = \bigl\{(r(e),e)\in E_U^0 \ \big|\ e\in E^1\bigr\}$ 
which is homeomorphic to $E^1$. 
Taking their complements, we get 
$(E_U^0)_{\rs{sg}} = 
\big\{(v,\infty) \in E_U^0 \ \big|\ v \in E^0\setminus U\big\}$ 
which is homeomorphic to $E^0\setminus U$. 

We define continuous maps $m_U^0$ and $m_U^1$ by 
\[
m_U^0\colon \tE_U^0\ni (v,e)\mapsto v\in \tE^0
\quad \text{ and } \quad 
m_U^1\colon \tE_U^1\ni (e',e)\mapsto e'\in \tE^1. 
\]
For $i=0,1$, 
$(\infty,\infty) \in \tE_U^i$ 
is the only element which is sent to $\infty\in \tE^i$ 
by $m_U^i$. 
Hence $m_U^i$ induces the proper continuous map 
from $E_U^i$ to $E^i$ for $i=0,1$. 
The following lemma is easy to see. 

\begin{lemma}\label{Lem:factormaprel}
The pair $m_U=(m_U^0,m_U^1)$ 
is a surjective factor map from $E_U$ to $E$ 
such that $(m_U^0)^{-1}(U) \subset (E_U^0)_{\rs{rg}}$.
\end{lemma}

\begin{proposition}\label{Prop:isomU}
The factor map $m_U=(m_U^0, m_U^1)$ 
induces an isomorphism $\mu\colon \cO(E;U)\to \cO(E_U)$. 
\end{proposition}

\begin{proof}
From the fact $(m_U^0)^{-1}(U) \subset (E_U^0)_{\rs{rg}}$, 
we get a \shom $\mu\colon \cO(E;U)\to \cO(E_U)$. 
This is injective because $m_U^0(E_U^0)_{\rs{sg}}=E^0\setminus U$ 
(see \cite[Corollary~11.8]{KaPim3}). 
One can see that this is surjective exactly 
same as in the proof of Proposition~\ref{Thm:isom1}. 
\end{proof}

For $k=1,2,\ldots$, 
we define a topological graph $E_{U,k}=(E_{U,k}^0,E_{U,k}^1,d_{U,k},r_{U,k})$ 
by
\begin{align*}
E_{U,k}^0 
&:= \bigl\{\, (v,\, e_1, e_2, \ldots, e_k) 
\in E^0 \times \tE^1 \times \cdots \times \tE^1
\ \big|\ (v, e_1) \in E_U^0, (e_i,e_{i+1}) \in \tE_U^1\bigr\} \\
E_{U,k}^1 
&:= \bigl\{ (e_0, e_1, e_2, \ldots, e_k) 
\in E^1 \times \tE^1 \times \cdots \times \tE^1
\ \big|\ (e_0, e_1) \in E_U^1, (e_i,e_{i+1}) \in \tE_U^1\bigr\}, 
\end{align*}
and 
\begin{align*}
d_{U,k}\big((e_0, e_1, e_2, \ldots, e_k)\big)
&:=(d(e_0), e_1, e_2, \ldots, e_k),\\ 
r_{U,k}\big((e_0, e_1, e_2, \ldots, e_k)\big)
&:=(r(e_0), e_0, e_1, \ldots, e_{k-1}) 
\end{align*}
for $(e_0, e_1, e_2, \ldots, e_k) \in E_{U,k}^1$. 
Then $E_{U,1}$ is nothing but $E_U$, and 
$E_{U,k+1}$ is the dual graph of $E_{U,k}$ for $k=1,2,\ldots$. 
Let $E_{U,0}$ be $E$. 
For $k,l\in\N$ with $k \leq l$, 
we define 
\begin{align*}
m_{k,l}^0 \colon 
E_{U,l}^0\ni (v,\, e_1, e_2, \ldots, e_l) 
& \mapsto (v,\, e_1, e_2, \ldots, e_k) \in E_{U,k}^0 \\
m_{k,l}^1 \colon 
E_{U,l}^1\ni (e_0, e_1, e_2, \ldots, e_l) 
& \mapsto (e_0, e_1, e_2, \ldots, e_k) \in E_{U,k}^1
\end{align*}
and set $m_{k,l} := (m_{k,l}^0, m_{k,l}^1)$. 
We have $m_{j,k}^i \circ m_{k,l}^i = m_{j,l}^i$ 
for $j \leq k \leq l$ and $i=0,1$. 
For $k \geq 1$, $m_{0,k}$ is simply denote $m_k$. 

For $1 \leq k \leq l$, $E_{U,k}$ and $m_{k,l}$ are 
the ones considered in Section~\ref{Sec:iterating} 
from $E_U=E_{U,1}$. 
Hence for $1 \leq k \leq l$, 
$m_{k,l}$ is a surjective regular factor map. 
On the other hand, 
one can easily show that for $1 \leq k$, 
$m_{k}$ is a surjective factor map with 
$(m_{k}^0)^{-1}(U) \subset (E_{U,k}^0)_{\rs{rg}}$. 

Next we define a topological graph 
$E_{U,\infty} = (E_{U,\infty}^0,E_{U,\infty}^1,d_{U,\infty},r_{U,\infty})$. 
by
\begin{align*}
E_{U,\infty}^0 
&:= \bigl\{\, (v,\, e_1, e_2, \ldots) 
\in E^0 \times \tE^1 \times \cdots 
\ \big|\ (v, e_1) \in E_U^0, (e_i,e_{i+1}) \in \tE_U^1\bigr\} \\
E_{U,\infty}^1 
&:= \bigl\{ (e_0, e_1, e_2, \ldots) 
\in E^1 \times \tE^1 \times \cdots 
\ \big|\ (e_0, e_1) \in E_U^1, (e_i,e_{i+1}) \in \tE_U^1\bigr\}. 
\end{align*}
and 
\begin{align*}
d_{U,\infty}\big((e_0, e_1, e_2, \ldots )\big)
&:=(d(e_0), e_1, e_2, \ldots ),\\ 
r_{U,\infty}\big((e_0, e_1, e_2, \ldots )\big)
&:=(r(e_0), e_0, e_1, \ldots ) 
\end{align*}
for $(e_0, e_1, e_2, \ldots) \in E_{U,\infty}^1$. 
We define factor maps $m_{k,\infty}$ 
from $E_{U,\infty}$ to $E_{U,k}$ for $k \in \N$ and 
a factor map $m_\infty$ from $E_{U,\infty}$ to $E$ 
in the same way as in Section~\ref{Sec:infinite}. 
Then similary as Theorem~\ref{Thm:Einfty=SGDS}
and Proposition~\ref{Prop:lim}, 
we can see that the topological graph 
$E_{U,\infty}$ is an SGDS. 
and that the topological graph $E_{U,\infty}$ 
coincided with the project limit of the 
projective system $((E_{U,k})_{k}, (m_{k,l})_{k,l})$. 

We have the following. 

\begin{theorem}\label{Thm:isomU}
The \shoms $\mu\colon \cO(E;U)\to \cO(E_{k,U})$ 
induced by the factor maps $m_k$ for $k \in \N$ and $k=\infty$
are isomorphisms. 
\end{theorem}

\begin{proof}
By Proposition~\ref{Prop:isomU}, 
the map from $\cO(E;U)$ to $\cO(E_U)$ is an isomorphism. 
By Theorem~\ref{Thm:isomEk} and Theorem~\ref{Thm:isomEinf}, 
the maps from $\cO(E_U)$ to $\cO(E_{k,U})$ for $k=1,2,\ldots$ 
and to $\cO(E_{\infty,U})$ 
are isomorphism. 
These facts finish the proof. 
\end{proof}

\begin{corollary}\label{Cor:gpoidEU}
The relative topological graph \Ca $\cO(E;U)$ is 
isomorphic to the \Ca $C^*(E_{\infty,U},\sigma)$ 
of the SGDS $(E_{\infty,U}^0,\sigma)$ 
where $\dom(\sigma)=r_{\infty,U}(E_{\infty,U}^1)$ 
and $\sigma = d_{\infty,U} \circ r_{\infty,U}^{-1}$. 
\end{corollary}

By this corollary, 
every relative topological graph \Ca $\cO(E;U)$ 
has the groupoid model $G(E_{\infty,U}^0,\sigma)$. 
This groupoid $G(E_{\infty,U}^0,\sigma)$ can be described 
in terms of graph $E$ as in Section~\ref{Sec:groupoid}. 
In particular, the unit space $E_{\infty,U}^0$ is, as a set, 
identified with the disjoint union $E_{\setminus U}^*\cup E^\infty$ 
where $E_{\setminus U}^* := \{\alpha \in E^* \mid d(\alpha) \in E^0\setminus U\}$. 
When $U=\emptyset$, we have $E_{\setminus \emptyset}^* =E^*$. 
Thus we recover the result \cite[Theorem~5.1]{Yee}.

\section{Partially defined topological graphs}\label{Sec:partial}

In this section, we introduce the notion of 
a partially defined topological graph, 
and investigate the \Ca associated to it. 
The content of this section can be applied to 
the situation in \cite[Section~7]{CK}.

\begin{lemma}\label{Lem:partialmap}
Let $E^0$ and $E^1$ be locally compact spaces. 
For a continuous map $r$ from an open subset $\dom(r)$ of $E^1$ to $E^0$, 
the following conditions are equivalent:
\begin{enumerate}
\rom
\item For every compact subset $X$ of $E^0$, 
its inverse image $r^{-1}(X)$ is closed in $E^1$. 
\item The map $\tr$ from $E^1$ to 
the one-point compactification $\tE^0$ of $E^0$ 
defined by 
\[
\tr(e)=\begin{cases} r(e) & \text{if $e \in \dom(r)$,} \\
\infty & \text{if $e \notin \dom(r)$} \end{cases}
\]
is continuous. 
\item There exist a locally compact space $U$ containing $E^0$ 
as an open subset and a continuous map $s$ from $E^1$ to $U$ 
such that $s^{-1}(E^0)=\dom(r)$ and $s|_{\dom(r)}=r$. 
\item For each $f \in C_0(E^0)$, the function $F$ on $E^1$ defined by 
\[
F(e)=\begin{cases} f(r(e)) & \text{if $e \in \dom(r)$,} \\
0 & \text{if $e \notin \dom(r)$} \end{cases}
\]
is continuous. 
\end{enumerate}
\end{lemma}

\begin{proof}
By the topology of one-point compactification, 
(i) implies (ii). 
It is trivial to see that (ii) implies (iii). 
Take $U$ and $s$ as in (iii). 
A function $f \in C_0(E^0)$ is naturally considered as an element in $C_0(U)$. 
Thus we can consider a continuous function $f \circ s$ on $E^1$ 
which is nothing but $F$ in (iv). 
Thus (iii) implies (iv). 

Finally assume (iv), and take a compact subset $X$ of $E^0$. 
There exists $f \in C_0(E^0)$ such that $f$ is $1$ on $X$. 
Consider the continuous function $F$ for $f$ as in (iv). 
Then $F^{-1}(1)$ is a subset of $\dom(r)$ containing $r^{-1}(X)$, 
and is a closed subset of $E^1$. 
These facts show $r^{-1}(X)$ is closed not only in $\dom(r)$ 
but also in $E^1$. 
We have shown that (iv) implies (i). 
\end{proof}

\begin{definition}
A {\em partially defined} topological graph is 
a quadruple $E=(E^0,E^1,d,r)$ where 
$E^0$ and $E^1$ are locally compact spaces, 
$d\colon E^1\to E^0$ is a local homeomorphism, 
and $r$ is a continuous map from $\dom(r)\subset E^1$ to $E^0$ 
satisfying the equivalent conditions
in Lemma \ref{Lem:partialmap}.
\end{definition}

From a partially defined topological graph $(E^0,E^1,d,r)$, 
one can obtain a C*-algebra $C_0(E^0)$ and a C*-correspondence $C_d(E^1)$ 
over $C_0(E^0)$ in the same way as from a topological graph (see \cite{KaI}). 
In fact the left action $\varphi\colon C_0(E^0) \to \cL(C_d(E^1))$ is defined 
through the map $C_0(E^0) \to C(E^1)$ defined 
in (iv) of Lemma \ref{Lem:partialmap}. 
When the domain $\dom(r)$ of $r$ is whole $E^1$, 
$E$ is a topological graph. 
We can define $\s{E}{0}{rg}\subset E^0$, 
and consider Cuntz-Krieger $E$-pairs $(t^0,t^1)$ 
on the same way as topological graphs. 
Thus we can define a \Ca $\cO(E)$ which is the \Ca associated with 
the C*-correspondence $C_d(E^1)$ in the sense of \cite{KaPim}. 

A partially defined topological graph $(E^0,E^1,d,r)$ 
and its \Ca $\cO(E)$ naturally arise as follows. 

\begin{proposition}\label{Prop:restrict}
Let $F=(F^0,F^1,d,r)$ be a topological graph, 
and $E^0$ be an open subset of $F^0$ with $d(F^1)\subset E^0$. 
Then the following hold. 
\begin{enumerate}
\rom
\item 
$E=(E^0,E^1,d,r)$ is a partially defined topological graph 
where $E^1 = F^1$, $\dom(r)=r^{-1}(E^0)$ and $d,r$ are restrictions 
of the ones of $F$. 
\item 
The \Ca $\cO(E)$ is isomorphic to the $C^*$-subalgebra $A$ of $\cO(F)$ 
generated by $t^0(C_0(E^0))$ and $t^1(C_d(E^1))$ 
where $(t^0,t^1)$ is the universal Cuntz-Krieger $F$-pair.
\item 
The \Ca $A$ in (ii) is the kernel of the surjection 
$\cO(F) \to C_0(\s{F}{0}{sg}\setminus E^0)$
\item 
If $\s{F}{0}{sg} \subset E^0$, 
then $\cO(E)$ is isomorphic to $\cO(F)$ 
\end{enumerate}
\end{proposition}

\begin{proof}
(i) It is clear that $E$ is a partially defined topological graph. 

(ii) It is clear that there is a \shom from $\cO(E)$ onto the $C^*$-subalgebra of $\cO(F)$ 
generated by $t^0(C_0(E^0))$ and $t^1(C_d(E^1))$. 
This \shom turns out to be isomorphism by 
Proposition~\ref{GIUT} below. 

(iii) 
The space $\s{F}{0}{sg}\setminus E^0$ consists of singular vertices, 
and there is no edge from a vertex in $\s{F}{0}{sg}\setminus E^0$ 
because $d(F^1)\subset E^0$. 
Hence there exists a surjection $\cO(F) \to C_0(\s{F}{0}{sg}\setminus E^0)$
whose kernel is generated by $t^0(C_0(E^0 \cup \s{F}{0}{rg}))$ 
and $t^1(C_d(E^1))$. 
Here note that 
$F_0\setminus (\s{F}{0}{sg}\setminus E^0)=E^0 \cup \s{F}{0}{rg}$. 
Since $C_0(\s{F}{0}{rg})$ is in the $C^*$-subalgebra of $\cO(F)$ 
generated by $t^1(C_d(E^1))$, 
it is in $A$ of (ii). 
Hence $A$ is the kernel of 
the surjection $\cO(F) \to C_0(\s{F}{0}{sg}\setminus E^0)$. 

(iv) 
This follows from (iii). 
\end{proof}

By \cite[Lemma~1.20]{KaI}, 
a C*-correspondence arising from a topological graph is always 
non-degenerate. 
On the other hand, 
a C*-correspondence arising from a partially defined topological graph 
$(E^0,E^1,d,r)$ 
is degenerate unless the domain $\dom(r)$ of $r$ is whole $E^1$. 
In this sense, we get new kinds of C*-correspondences 
from partially defined topological graphs. 
Nevertheless many results for C*-algebras of topological graphs 
can be similarly applied to this case. 
This is the case particularly for the results obtained through 
the theory of Pimsner algebras. 
The following are examples of such results. 
See \cite{KaI} for notation and proofs. 
(See also \cite{KaPim2} for proofs.) 

\begin{proposition}[{see \cite[Theorem~4.5]{KaI}}]\label{GIUT}
For a partially defined topological graph $E=(E^0,E^1,d,r)$ and 
a Cuntz-Krieger $E$-pair $T=(T^0,T^1)$, 
the following are equivalent:
\begin{enumerate}
\rom
\item The map $\rho_{T}\colon \cO(E)\to C^*(T)$ is an isomorphism.
\item The map $T^0$ is injective and 
there exists an automorphism $\beta'_z$ of $C^*(T)$ 
such that $\beta'_z(T^0(f))=T^0(f)$ and $\beta'_z(T^1(\xi))=zT^1(\xi)$
for every $z\in\T$.
\item The map $T^0$ is injective and 
there exists a conditional expectation $\varPsi_T$ 
from $C^*(T)$ onto $\F_T$ such that 
$\varPsi_T(T^n(\xi)T^m(\eta)^*)=\delta_{n,m}T^n(\xi)T^m(\eta)^*$ 
for $\xi\in C_d(E^n)$ and $\eta\in C_d(E^m)$.
\end{enumerate}
\end{proposition}

\begin{proposition}[{see \cite[Proposition~6.1]{KaI}}]
For a partially defined topological graph $E$, 
the $C^*$-algebra $\cO(E)$ is nuclear. 
\end{proposition}

\begin{proposition}[{see \cite[Proposition~6.3 and Proposition~6.6]{KaI}}]
For a partially defined topological graph $E=(E^0,E^1,d,r)$,
the $C^*$-algebra $\cO(E)$ is separable 
if and only if both $E^0$ and $E^1$ are second countable. 
In this case, $\cO(E)$ satisfies the UCT. 
\end{proposition}

\begin{proposition}[{see \cite[Corollary~6.10]{KaI}}]
For a partially defined topological graph $E=(E^0,E^1,d,r)$, 
we have the following exact sequence of $K$-groups: 
\[
\begin{CD}
K_0(C_0(\s{E}{0}{rg})) @>>\iota_*-[\pi_r]> K_0(C_0(E^0)) 
@>>t^0_*>  K_0(\cO(E)) \\
@AAA @. @VVV \\
K_1(\cO(E)) @<t^0_*<< K_1(C_0(E^0)) 
@<\iota_*-[\pi_r]<< K_1(C_0(\s{E}{0}{rg})).
\end{CD}
\]
\end{proposition}

By \cite[Proposition~7.1]{KaII}, 
for a topological graph $E=(E^0,E^1,d,r)$, 
the C*-algebra $\cO(E)$ is unital if and only if $E^0$ is compact. 
When $r$ is partially defined, the situation becomes very complicated. 
We give four examples.
In these examples, 
Proposition~\ref{GIUT} is useful to determine the \CA s. 

\begin{example}\label{Ex1}
Consider a partially defined topological graph $E=(E^0,E^1,d,r)$ 
where $E^0=\{v,w\}$, $E^1=\{e,f\}$, $d(e)=v$, $d(f)=w$, $\dom(r)=\{e\}$ 
and $r(e)=w$. 
The C*-algebra $\cO(E)$ is generated by two orthogonal projections 
$p_v, p_w$ and two partial isometries $t_e, t_f$ such that 
\[
t_e^*t_e = p_v,\ t_f^*t_f = p_w,\ p_w = t_et_e^*,\ p_vt_f=p_wt_f=0. 
\]
Then one can check that 
$p_v+p_w+t_ft_f^*$ becomes the unit of $\cO(E)$. 
In fact, one can see that $\cO(E)$ is isomorphic to $M_3(\C)$. 
\end{example}

\begin{example}\label{Ex2}
Next consider a partially defined topological graph $E=(E^0,E^1,d,r)$ where $E^0=\{v,w\}$, 
$E^1=\{e_1,e_2,\ldots,e_k,\ldots \}$, $d(e_k)=v$ for $k=1,2,\ldots$, 
$\dom(r)=\{e_1\}$ and $r(e_1)=w$. 
The C*-algebra $\cO(E)$ is generated by two orthogonal projections 
$p_v, p_w$ and partial isometries $t_k$ for $k=1,2,\ldots$ such that 
\[
t_k^*t_k = p_v\ (k=1,2,\ldots),\ p_w = t_1t_1^*,\ 
p_vt_k=p_wt_k=0\ (k=2,3,\ldots). 
\]
One can check that $\cO(E)$ is isomorphic to the non-unital C*-algebra 
$\cK$ of compact operators on $\ell^2(\N)$. 
In fact, we get an isomorphism from $\cO(E)$ to $\cK$ sending 
$p_v$ and $p_w$ to the matrix units $w_{0,0}$ and $w_{1,1}$
for $(0,0)$ and $(1,1)$ positions respectively, 
and $t_k$ to the matrix units $w_{k,0}$ for $(k,0)$ positions. 
\end{example}

\begin{example}\label{Ex3}
Let us consider a partially defined topological graph $E=(E^0,E^1,d,r)$ 
where $E^0=[0,1)\cup (1,2)$ and $E^1=[0,1)$,
$d\colon E^1 \to E^0$ is the embedding, 
$\dom(r)=(0,1)$ and $r(s)=s+1$ for $s \in \dom(r)$. 

Let us define a $*$-homomorphism 
$t^0\colon C_0(E^0)\to C_0([0,1),M_2(\C))$ by 
\[
t^0(f)(s)=
\begin{pmatrix}
f(s) & \\
 & f(s+1) 
\end{pmatrix}
\]
for $s \in [0,1)$ and $f \in C_0(E^0)$. 

Let us also define a linear map 
$t^1 \colon C_d(E^1)\to C_0([0,1),M_2(\C))$ by 
\[
t^1(\xi)(s)=\begin{pmatrix}
0 & \\
\xi(s) & 0 
\end{pmatrix}
\]
for $s \in [0,1)$ and $\xi \in C_d(E^1) = C_0([0,1))$. 
One can check that the pair of maps $(t^0,t^1)$ induces an isomorphism 
$\cO(E) \to C_0([0,1),M_2(\C))$. 
Thus $\cO(E)$ is non-unital. 
\end{example}

\begin{example}\label{Ex4}
Let us consider a partially defined topological graph $E=(E^0,E^1,d,r)$ 
where $E^0=(-1,1)$ and $E^1=(-1/2,1/2)$ are open intervals,
$d\colon E^1 \to E^0$ is the embedding, 
$\dom(r)=E^1\setminus\{0\}$ and $r$ is defined by 
\[
r(s)=\begin{cases}
s+1 & \text{for $s \in (-1/2, 0)$,}\\
s-1 & \text{for $s \in (0, 1/2)$.}
\end{cases}
\]
Let us define a $*$-homomorphism 
$t^0\colon C_0(E^0)\to C([-1/2,1/2],M_2(\C))$ by 
\[
t^0(f)(s)=
\begin{cases}
\begin{pmatrix}
f(s) & \\
 & f(s+1) 
\end{pmatrix}& \text{for $s \in [-1/2, 0]$}\\
\begin{pmatrix}
f(s) & \\
 & f(s-1)
\end{pmatrix}& \text{for $s \in [0, 1/2]$}
\end{cases}
\]
for $f \in C_0(E^0) = C_0(-1,1)$. 

Let us also define a linear map 
$t^1 \colon C_d(E^1)\to C([-1/2,1/2],M_2(\C))$ by 
\[
t^1(\xi)(s)=\begin{pmatrix}
0 & \\
\xi(s) & 0 
\end{pmatrix}\qquad \text{for $s \in [-1/2, 1/2]$}
\]
for $\xi \in C_d(E^1) = C_0(-1/2,1/2)$. 
One can check that the pair of maps $(t^0,t^1)$ induces an injective 
$*$-homomorphism $\cO(E) \to C([-1/2,1/2],M_2(\C))$. 
The image of this $*$-homomorphism consists of 
$F \in C([-1/2,1/2],M_2(\C))$ such that 
\begin{align*}
F(-1/2) &= 
\begin{pmatrix}
a & \\
 & b \\
\end{pmatrix}&
F(1/2) &= 
\begin{pmatrix}
b & \\
 & a
\end{pmatrix}
\end{align*}
for some $a,b \in \C$. 
It is now easy to see that this image is unital. 
Hence $\cO(E)$ is unital. 
The unit is given for example by $t^0(f)+t^1(\xi)t^1(\xi)^*$ 
where $f \in C_0(E^0) = C_0(-1,1)$ and $\xi \in C_d(E^1) = C_0(-1/2,1/2)$ are 
defined as 
\[
f(s)= \begin{cases}
\sin^2(s\pi)& \text{for $s \in [-1, -1/2]$}\\
1& \text{for $s \in [-1/2, 1/2]$}\\
\sin^2(s\pi)& \text{for $s \in [1/2, 1]$}
\end{cases}
\qquad
\xi(s)= \cos(s\pi)\ \text{for $s \in [-1/2, 1/2]$.}
\]
\end{example}

The four examples above show that the unitality of $\cO(E)$ is 
independent of the compactness of $E^0$ 
when $r$ is partially defined. 
It is relevant to the unitality of $\cO(E)$ that 
the compactness of $\s{E}{0}{sg}$, the compactness of $E^1\setminus \dom(r)$ 
and the map $r$ itself. 
In the end of this section, 
we give equivalent conditions on $E$ for $\cO(E)$ to be unital. 

Next we try to find the groupoid model for $\cO(E)$. 
Take a partially defined topological graph $E=(E^0,E^1,d,r)$. 
By Lemma \ref{Lem:partialmap}, 
we get the continuous map $\tr$ from $E^1$ to 
the one-point compactification $\tE^0$ of $E^0$ 
such that $\tr^{-1}(E^0)=\dom(r)$ and $\tr|_{\dom(r)}=r$. 
Then the quadruple $\tE :=(\tE^0,E^1,d,\tr)$ becomes a topological 
graph. 
The point $\infty \in \tE^0$ can be regular or singular 
in this topological graph $\tE$. 
The set $\s{\tE}{0}{rg}$ of regular points in $\tE$ 
is either $\s{E}{0}{rg}$ or $\s{E}{0}{rg} \cup \{\infty\}$. 

We get the following. 

\begin{proposition}\label{Prop:unitize}
The evaluation map at $\infty \in \tE^0$ induces 
the unital $*$-homomorphism $\chi\colon \cO(\tE;\s{E}{0}{rg}) \to \C$ 
whose kernel is naturally isomorphic to $\cO(E)$. 
Therefore $\cO(\tE;\s{E}{0}{rg})$ is the unitization of $\cO(E)$. 
\end{proposition}

\begin{proof}
Since $\infty \notin \s{E}{0}{rg}$, 
one can define a unital $*$-homomorphism 
$\chi\colon \cO(\tE;\s{E}{0}{rg}) \to \C$ 
whose kernel is generated by $t^1(C_d(E^1))$ and $t_0(C_0(E^0))$ 
where $(t^0,t^1)$ is the universal Cuntz-Krieger $\tE$-pair. 
It is routine to check that this kernel is isomorphic to $\cO(E)$. 
\end{proof}

When $\infty \in \tE^0$ is singular in $\tE$, 
we have $\cO(\tE;\s{E}{0}{rg})=\cO(\tE)$, 
and hence $\cO(\tE)$ is the unitization of $\cO(E)$. 
In this case, Proposition~\ref{Prop:unitize} 
follows from Proposition~\ref{Prop:restrict} (iv). 
When $\infty \in \tE^0$ is regular in $\tE$, 
we have the following. 

\begin{proposition}\label{Prop:coincide}
When $\infty \in \tE^0$ is regular in $\tE$, 
the \Ca $\cO(E)$ is isomorphic to $\cO(\tE)$
\end{proposition}

\begin{proof}
This follows from Proposition~\ref{Prop:restrict} (iv). 
Here we give another proof using Proposition~\ref{Prop:unitize}. 

Let $(t^0,t^1)$ be the universal pair 
defining $\cO(\tE;\s{E}{0}{rg})$. 
Take $f \in C_0(\s{\tE}{0}{rg})$ with $f(\infty)=1$, 
and set $p=t^0(f)-\psi_{t^1}(\pi_r(f))$ 
where $\psi_{t^1} \colon \cK(C_d(E^1)) \to \cO(\tE;\s{E}{0}{rg})$ is 
the natural map defined using $t^1$. 
We note that $p$ does not depend on the choice of $f$. 

The kernel of the natural surjection 
$\cO(\tE;\s{E}{0}{rg}) \to \cO(\tE)$ is spanned by $p$. 
We are going to show that the restriction 
$\iota \colon \cO(E) \to \cO(\tE)$ of this surjection 
to $\cO(E) \subset \cO(\tE;\s{E}{0}{rg})$ 
is an isomorphism 
here $\cO(E)$ is naturally considered as a subalgebra 
of $\cO(\tE;\s{E}{0}{rg})$. 
Since $p \notin \cO(E)$, 
$\iota$ is injective. 
Since $\cO(E) + \C p =\cO(\tE;\s{E}{0}{rg})$, 
$\iota$ is surjective. 
We are done. 
\end{proof}

\begin{corollary}\label{Cor:unital}
When $\infty \in \tE^0$ is regular in $\tE$, 
the \Ca $\cO(E)$ is unital. 
\end{corollary}

\begin{proof}
Since $\tE$ is a topological graph such that $\tE^0$ is compact, 
$\cO(\tE)$ is unital. 
By Proposition~\ref{Prop:coincide}, 
when $\infty \in \tE^0$ is regular in $\tE$, 
the \Ca $\cO(E)$ is isomorphic to $\cO(\tE)$ 
and hence is unital. 
\end{proof}

It turns out that for a partially defined topologicla graph $E$ 
with $\dom (r) \neq E^1$, the converse of this corollary is true 
(Theorem~\ref{Thm:unital}). 
We can apply this proposition and this corollary to 
Example~\ref{Ex1} and Example~\ref{Ex4}. 

Using Proposition~\ref{Prop:unitize}, 
we can get a groupoid model for $\cO(E)$. 
By Corollary~\ref{Cor:gpoidEU}, 
the \Ca $\cO(\tE;\s{E}{0}{rg})$ is isomorphic to the \Ca 
$C^*(X,\sigma)$ of the SGDS $(X,\sigma)$ where 
$X = \tE_{\infty,\s{E}{0}{rg}}^0$, 
$\dom(\sigma)=r_{\infty,\s{E}{0}{rg}}(E_{\infty,\s{E}{0}{rg}}^1)$ 
and $\sigma = d_{\infty,\s{E}{0}{rg}} \circ r_{\infty,\s{E}{0}{rg}}^{-1}$. 
Recall that 
\[
X = \bigl\{\, (v,\, e_1, e_2, \ldots) 
\in \tE^0 \times \tE^1 \times \cdots 
\ \big|\ (v, e_1) \in \tE_\s{E}{0}{rg}^0, (e_i,e_{i+1}) \in 
\widetilde{\tE_\s{E}{0}{rg}^1}\bigr\}
\]
where 
\begin{align*}
\tE_\s{E}{0}{rg}^0 
&= \bigl\{(v\,,e)\in \tE^0\times \tE^1\ \big|\ 
\text{$v=\tr(e)$ if $e\in E^1$, 
$v\in \tE^0\setminus \s{E}{0}{rg}$ if $e=\infty$}\bigr\},\\
\widetilde{\tE_\s{E}{0}{rg}^1} 
&= \bigl\{(e',e)\in \tE^1 \times \tE^1\ \big|\ 
\text{$d(e')=\tr(e)$ if $e\in E^1$, 
$e'\in\tE^1\setminus d^{-1}(\s{E}{0}{rg})$ if $e=\infty$}\bigr\}. 
\end{align*}
Here note that a topological graph we consider now is 
$\tE=(\tE^0,E^1,d,\tr)$. 
Thus $\tE_\s{E}{0}{rg}^0$ is not $\widetilde{E_{\s{E}{0}{rg}}^0}$, but
$(\tE)_{\s{E}{0}{rg}}^0$. 
Note also that $\tE^0\setminus \s{E}{0}{rg} = \s{E}{0}{sg} \cup \{\infty\}$. 
Since $\tE^0$ is compact, the \Ca $\cO(\tE;\s{E}{0}{rg})$ is unital. 
Hence $X$ is compact. 
We define 
\[
Y := \bigl\{(v\,,e)\in \tE^0\times \tE^1\ \big|\ 
\text{$v=\tr(e)$ if $e\in E^1$, 
$v\in \s{E}{0}{sg}$ if $e=\infty$}\bigr\},
\]
$Z := \widetilde{\tE_\s{E}{0}{rg}^1}$ and 
\[
W :=\bigl\{\, (v,\, e_1, e_2, \ldots) 
\in \tE^0 \times \tE^1 \times \cdots 
\ \big|\ (v, e_1) \in Y,\ (e_i,e_{i+1}) \in Z \bigr\}. 
\]
Note that $X = W \cup \{(\infty, \infty, \ldots)\}$, 
and hence $X$ is the one-point compactification of $W$. 
Since the images of $d_{\infty,\s{E}{0}{rg}}$ and 
$r_{\infty,\s{E}{0}{rg}}$ are contained in $W$, 
we get an SGDS $(W,\sigma)$ in the same way as 
$(X,\sigma)$. 
Note that the \'etale groupoid $G(W,\sigma)$ is the 
restriction of $G(X,\sigma)$ to $W \subset X$. 

\begin{proposition}\label{Prop:Wsigma}
The \Ca $\cO(E)$ is isomorphic to the \Ca $C^*(W,\sigma)$
of SGDS $(W,\sigma)$. 
\end{proposition}

\begin{proof}
The \Ca $C^*(W,\sigma)=C^*(G(W,\sigma))$ is the kernel of the surjection 
from $C^*(X,\sigma)=C^*(G(X,\sigma))$ to $\C$ obtained from 
evaluating at $(\infty, \infty, \ldots)$. 
Through the isomorphism between $C^*(X,\sigma)$ and $\cO(\tE;\s{E}{0}{rg})$, 
this surjection can be identified with the surjection 
in Proposition~\ref{Prop:unitize}. 
Therefore $C^*(W,\sigma)$ is isomorphic to $\cO(E)$. 
\end{proof}

By this proposition, the \Ca $\cO(E)$ has the 
groupoid model $G(W,\sigma)$. 
This groupoid $G(W,\sigma)$ can be described 
in terms of graph $E$ as in Section~\ref{Sec:groupoid}. 
In particular, the unit space $W$ is, as a set, 
identified with $\s{E}{*}{sg} \cup E^\infty$ 
as in Section~\ref{Sec:groupoid}. 
Thus the groupoid $G(W,\sigma)$ coincides with the one 
in \cite[Corollary~7.11]{CK} considered by Castro and Kang 
(see also \cite[Definition~7.5]{CK}). 

Finally, we get the following. 

\begin{theorem}\label{Thm:unital}
Let $E=(E^0,E^1,d,r)$ be a partially defined topological graph 
with $\dom(r) \neq E^1$. 
The following coonditions are equivalent:
\begin{enumerate}
\rom
\item $\cO(E)$ is unital,
\item $\infty$ is regular in $\tE$, 
\item the space $Y$ is compact,
\item the space $W$ is compact,
\item the point $(\infty, \infty, \ldots)$ is isolated in $X$. 
\end{enumerate}
\end{theorem}

\begin{proof}
A \Ca $C^*(G)$ of an \'etale groupoid $G$ is unital 
if and only if the unit space of $G$ is compact. 
Hence (i) is equivalent to (iv) by Proposition~\ref{Prop:Wsigma}. 
It is clear that (iv) is equivalent to (v). 
Since there exists a surjection $(v,\, e_1, e_2, \ldots) \mapsto (v,e_1)$ 
from $W$ to $Y$, (iv) implies (iii). 

Assume (iii), and we show (ii). 
Since $\{(v,\infty)\in Y\mid v \in \s{E}{0}{sg}\}$ is closed in $Y$, 
$\s{E}{0}{sg}$ is compact. 
Hence we have a compact neighbourhood $C$ of $\infty \in \tE^0$ 
such that $C \cap \s{E}{0}{sg} = \emptyset$. 
Since 
\[
\{(v,e)\in Y\mid v \in C\}
=\{(v,e)\in \tE^0\times E^1\mid v=\tr(e) \in C\}
\]
is closed in $Y$, 
$\tr^{-1}(C)$ is compact. 
Since $\tr^{-1}(\infty)=E^1 \setminus \dom(r) \neq \emptyset$ 
and $C \setminus \{\infty\} \subset \s{E}{0}{rg}$, 
we have $\tr(\tr^{-1}(C))=C$. 
Therefore $\infty$ is regular in $\tE$ by \cite[Proposition~2.8]{KaI}. 
This shows that (iii) implies (ii). 
Finally by Corollary~\ref{Cor:unital}, (ii) implies (i). 
\end{proof}

By (iii) of the proposition above, 
the compactness of $\s{E}{0}{sg}$ and $E^1\setminus \dom(r)$ is 
a necessary condition for $\cO(E)$ to be unital.
This is not a sufficient condition as the following example shows. 

\begin{example}
Consider a partially defined topological graph $E=(E^0,E^1,d,r)$ 
where $E^0=\{v_0,v_1,v_2,\ldots\}$, 
$E^1=\{e_1,e_2,\ldots \}$, $d(e_k)=v_k$ for $k=1,2,\ldots$, 
$\dom(r)=\{e_2,e_3,\ldots \}$ and $r(e_k)=v_{k-1}$ for $k=1,2,\ldots$. 
The C*-algebra $\cO(E)$ is generated by orthogonal projections 
$p_k$ and partial isometries $t_k$ for $k=1,2,\ldots$ such that 
\[
t_k^*t_k = p_k\ (k=1,2,\ldots),\ p_{k-1} = t_kt_k^*\ (k=2,3,\ldots),\ 
p_kt_1=0\ (k=0,1,\ldots). 
\]
Then $\s{E}{0}{sg}=\{v_0\}$ and $E^1\setminus \dom(r)=\{e_1\}$ are compact. 
On the other hand, 
$\cO(E)$ is isomorphic to the non-unital C*-algebra $\C \oplus \cK$. 
In fact, we get an isomorphism from $\cO(E)$ to $\C \oplus \cK$ sending 
$p_0$ to $(1,0)$, $p_k$ to $(0,w_{k,k})$ for $k=1,2,\ldots$ 
and $t_k$ to $(0,w_{k-1,k})$ for $k=1,2,\ldots$ 
\end{example}

\end{document}